\documentclass[reqno,oneside]{amsart}

\usepackage{hyperref}
\usepackage{amssymb,amsmath,amssymb,amsfonts,amsthm}
\usepackage{cases}
\usepackage{color}
\usepackage{empheq}
\usepackage{url}
\usepackage{geometry}
\usepackage{cite,mathtools}

\numberwithin{equation}{section}

\theoremstyle{plain}
\newtheorem{theorem}{Theorem}[section]
\newtheorem{lemma}{Lemma}[section]
\newtheorem{proposition}{Proposition}[section]

\theoremstyle{definition}

\newtheorem{example}{Example}[section]
\newtheorem{remark}{Remark}[section]

\allowdisplaybreaks[4]

\def\BE{\mathbb E}

\def\BR{\mathbb R}

\def\rd{\mathrm d}
\def\rdiv{\mathrm{div}}
\def\e{\mathrm e}
\def\supp{\mathrm{supp}}

\def\La{\Lambda}
\def\Om{\Omega}
\def\al{\alpha}
\def\be{\beta}
\def\ga{\gamma}
\def\de{\delta}
\def\ep{\epsilon}
\def\ve{\varepsilon}

\def\te{\theta}
\def\ka{\kappa}
\def\la{\lambda}

\def\vp{\varphi}

\def\f{\frac}
\def\nb{\nabla}
\def\ov{\overline}
\def\pa{\partial}
\def\wt{\widetilde}

\title[Energy decay of nonlocal viscoelastic equations]{Energy decay of nonlocal viscoelastic equations with nonlinear damping and polynomial nonlinearity}

\author[Qingqing Peng]{Qingqing Peng$^{1,2}$}
\thanks{$^1$School of Mathematics and Statistics \& Hubei Key Laboratory of Engineering Modeling and Scientific Computing, Huazhong University of Science and Technology, Wuhan 430074, China.}
\thanks{$^2$Department of Mathematics, Kyoto University, Kitashirakawa-Oiwakecho, Sakyo-ku, Kyoto 606-8502, Japan.}

\author[Yikan Liu]{Yikan Liu$^{2,*}$}
\thanks{$^*$Corresponding author.  E-mail: {\tt liu.yikan.8z@kyoto-u.ac.jp}}


\begin{document}

\begin{abstract}
This paper is concerned with the energy decay of a viscoelastic variable coefficient wave equation with nonlocality in time as well as nonlinear damping and polynomial nonlinear terms. Using the Lyapunov method, we establish a polynomial energy decay for the solution under relatively weak assumptions regarding the kernel of the nonlocal term. More specifically, we improve the decay rate of the energy by additionally imposing a certain convexity assumption on the kernel. Several examples are provided to confirm such improvements to faster polynomial or even exponential decays.
\end{abstract}

\maketitle	

\section{Introduction}
Let $\BR_+:=(0,\infty)$ and $\Om\subset\BR^n$ ($n\ge1$) be a bounded domain with a smooth boundary $\pa\Om$. The main purpose of this manuscript is to investigate the energy decay rate of the following initial-boundary value problem for a nonlocal viscoelastic wave equation with variable coefficients, a nonlinear damping term and a polynomial nonlinear term:
\begin{equation}\label{1.1}
\left\{\begin{alignedat}{2}
& \begin{aligned}
& u_{tt}-\rdiv(A(x)\nb u)+\int_0^t f(t-s)\,\rdiv(a(x)\nb u(s))\,\rd s\\
& \quad+b(x)h(u_t)=k(x)|u|^{p-2}u
\end{aligned} & \quad &\mbox{in }\Om\times\BR_+,\\
& u=u^0,\ u_t=u^1 & \quad &\mbox{in }\Om\times\{0\},\\
& u=0 & \quad &\mbox{on }\pa\Om\times\BR_+,
\end{alignedat}\right.
\end{equation}
where the initial displacement $u^0\in H_0^1(\Om)$ and the initial velocity $u^1\in L^2(\Om)$. The principal coefficient $A(x)$ is a symmetric and strictly positive-definite matrix-valued function on $\ov\Om$. The nonlocal term $\int_0^t f(t-s)\,\rdiv(a(x)\nb u(s))\,\rd s$ represents the memory effect in time, which takes the form of a convolution of the solution $u$ up to the second derivatives in space with the kernel function $f(t)$. Both of the nonlinear damping $h(u_t)$ and polynomial nonlinear term $|u|^{p-2}u$ are equipped with $x$-dependent coefficients $b(x)$ and $k(x)$ respectively, among which we assume
$$
k\in C^1(\ov\Om),\quad k\ge0\mbox{ on }\ov\Om.
$$
Moreover, the exponent constant $p$ satisfies
$$
\left\{\begin{alignedat}{2}
& p>2, & \quad & n=1,2,\\
& 2< p\le\f{2n-2}{n-2}, & \quad & n\ge3.
\end{alignedat}\right.
$$
Detailed assumptions on $A(x),f(t),a(x),b(x)$ will be specified in Section \ref{sec-prewell}.

Over the past few decades, significant results have been established regarding well-posedness and energy decays of nonlinear wave equations with either weak or strong damping. As a pioneering work, Webb \cite{W1} first investigated the well-posedness of solutions through the operator theory and examined the long-term behavior of these solutions using Lyapunov stability techniques for the semilinear wave equation with strong damping and a forcing term. Gazzola and Squassina \cite{G1} studied the semilinear wave equation with strong frictional damping and established both existence and nonexistence results for solutions when the initial energy does not exceed the depth of the potential well. Further, they discovered the blow-up phenomena by potential well methods along with Levine's concavity techniques. Subsequently, Ma and Fang \cite{M1} demonstrated both existence and nonexistence of global weak solutions and derived energy decay estimates by combining potential well methods with a logarithmic Sobolev inequality applicable to the viscoelastic wave equations. Lian and Xu \cite{L1} further explored global existence, energy decay and infinite-time blow-up using methodologies similarly to \cite{M1}. Di, Shang and Song \cite{D1} discussed the semilinear wave equation with logarithmic nonlinearity and strong damping. They proved global well-posedness along with polynomial or exponential stability through potential methods combined with Lyapunov skills, and also considered various blow-up scenarios. Cui and Chai \cite{cui} investigated the energy decay for variable coefficient wave equations with logarithmic nonlinearity by applying vector assumptions and multiplicative methods. Moreover, Yang and Zhong \cite{Y1} investigated the global well-posedness and blow-up phenomena for a strongly damped wave equation with a time-varying source and singular dissipation. They established both local well-posedness and globally existing solutions based on cut-off techniques, multiplier methods, contraction mapping principles and modified potential approaches. Meanwhile, they also derived the blow-up results of solutions with arbitrary positive initial energy and the lifespan of the blow-up solutions are derived.

Incorporating the nonlocal effect in time, viscoelastic wave equations with frictional damping have also witnessed significant developments concerning the energy decay and blow-up phenomena. In \cite{C1}, Cavalcanti and Oquendo examined a viscoelastic equation with frictional damping as follows:
\[
u_{tt}-\triangle u+\int_0^t g(t-s)\,\rdiv(a(x)\nb u(s))\,\rd s+b(x)h(u_t)+f(u)=0.
\]
They derived decay rates when the kernel function $g(s)$ and the coefficients $a(x),b(x)$ satisfy certain conditions. Dias et al. \cite{D3} explored the energy decay of a similar wave equation in an open unbounded domain, where the coefficients $a(x),b(x)$ meet the aforementioned conditions and the kernel function satisfies some convexity conditions. Furthermore, Cavalcanti et al. \cite{C2} established stability results for an equation on a compact Riemannian manifold, with similar assumptions on $g,a,b$ similarly to those in \cite{C1}. Jin et al. \cite{22} discussed uniform exponential and polynomial decay rates for solutions under weaker conditions than those in \cite{C1}. Mustafa and Abusharkh \cite{M2} investigated the energy decay of plate equations with frictional and viscoelastic damping, obtaining decay results under assumptions similarly to those in \cite{D3}. In particular, Li and Gao \cite{F1} used Lyapunov methods to derive blow-up results for viscoelastic equations and provided estimates for blow-up times. Ha and Park \cite{H1} also established blow-up results and local existence of solutions for viscoelastic equations with logarithmic nonlinearity using potential well methods and Lyapunov techniques. Additionally, Hao and Du \cite{H2} studied the energy decay and blow-up for viscoelastic equations with variable coefficients and logarithmic nonlinearity, employing potential well and Lyapunov methods akin to those in \cite{H1}.

However, to the best of our knowledge, there are limited researches on a viscoelastic equation with variable coefficients that incorporates a frictional damping term $b(x)h(u_t)$ and a nonlinear term $k(x)|u|^{p-2}u$. Motivated by the literature revisited above, in this article we investigate the energy decay rates of solutions to viscoelastic equations taking the form of \eqref{1.1}. We establish energy decay results under weaker assumptions on the kernel function $f(t)$ compared to previous works such as \cite{C1,C2,H1,H2}. In this sense, our findings improve the aforementioned results.

The remainder of this paper is organized as follows. In Section \ref{sec-prewell}, we specify assumptions on coefficients and prepare preliminary lemmas including some elementary estimates of energies the global well-posedness of solutions. In Section \ref{sec-main}, we state the two main results of this article, namely, the polynomial energy decay in Theorem \ref{theorem3.2} as well as a refined decay estimate in Theorem \ref{the3.2} under a weak convexity assumption on the kernel function $f(t)$. Further, we construct several examples to illustrate that Theorem \ref{the3.2} does imply faster decay rates in some cases. Subsequently, Sections \ref{sec-proof1}--\ref{sec-proof2} are devoted to the proofs of Theorems \ref{theorem3.2}--\ref{the3.2}, respectively.

\section{Preliminary and well-posedness}\label{sec-prewell}

In this section, we fix notations and introduce some basic definitions, important lemmas and function spaces for the sake of the statements and proofs of our main results.

Throughout this article, we denote the norm of the Lebesgue space $L^p(\Om)$ ($1\le p\le\infty$) by $\|\cdot\|_p$, and the inner product of $L^2(\Om)$ by $(\,\cdot\,,\,\cdot\,)$.

\subsection{Some assumptions and definitions}

We start with giving assumptions on the coefficients involved in the governing equation of \eqref{1.1}.\medskip

{\bf(A1)} The function $A(x)$ satisfies $A=(a_{ij})_{1\le i,j\le n}\in C^\infty(\ov\Om;\BR^{n\times n})$ and there exist constants $\mu_0>\la_0>0$ such that
\begin{equation}\label{1.2}
\la_0|\xi|^2\le A(x)\xi\cdot\,\xi\le\mu_0|\xi|^2,\quad\forall\,x\in\ov\Om,\ \forall\,\xi\in\BR^n.
\end{equation}

{\bf(A2)} The functions $a(x)$ and $b(x)$ satisfy
$$
0\le a\in C^1(\ov\Om),\quad0\le b\in L^\infty(\Om),\quad a\not\equiv0\mbox{ on }\pa\Om
$$
and there exists a constant $\ka>0$ such that
\[
a+b\ge\ka>0\mbox{ in }\Om.
\]

{\bf(A3)} The function $h\in C(\BR)$ is nondecreasing and satisfies
$$
h(s)s\ge0,\quad\forall\,s\in\BR.
$$
Moreover, there exist constants $c_i>0$ ($i=1,2,3,4$) and $q\ge1$ such that
\[
\begin{cases}
c_1|s|\le|h(s)|\le c_2|s|, & |s|>1,\\
c_3|s|^q\le|h(s)|\le c_4|s|^{1/q}, & |s|\le1.
\end{cases}
\]

{\bf(A4)} The function $f\in C^1(\ov{\BR_+};\BR_+)$ is nonincreasing and satisfies
\begin{equation}\label{2.1}
f(0)>0,\quad\la_0-\|a\|_\infty\int_0^\infty f(t)\,\rd t:=\ell>0.
\end{equation}

Next, we define
\begin{equation}\label{eq-def-FM}
F(t):=\int_t^\infty f(s)\,\rd s,\quad M(\de):=\int_0^\infty\f{f(s)}{K_\de(s)}\,\rd s,\quad K_\de(s):=-\f{f'(s)}{f(s)}+\de,
\end{equation}
where $\de\in(0,1)$ is a constant.

We introduce the energy functional as follows
\begin{align}
E(t) & :=\f12\|u_t(t)\|_2^2+\f12\int_\Om A\nb u(t)\cdot\nb u(t)\,\rd x-\f12\int_0^t f(s)\,\rd s\int_\Om a|\nb u(t)|^2\,\rd x\nonumber\\
& \quad\:\,+\f12(f\circ\nb u)(t)-\f1p\int_\Om k|u(t)|^p\,\rd x,\label{2.5}
\end{align}
where
\begin{equation}\label{eq-f-grad}
(f\circ\nb u)(t):=\int_0^t f(t-s)\int_\Om a|\nb(u(t)-u(s))|^2\,\rd x\rd s\ge0.
\end{equation}
Here owing to assumptions (A1) and (A4), we estimate
\begin{align}
& \quad\,\int_\Om A\nb u(t)\cdot\nb u(t)\,\rd x-\int_0^t f(s)\,\rd s\int_\Om a|\nb u(t)|^2\,\rd x\nonumber\\
& \ge\int_\Om\la_0|\nb u(t)|^2\,\rd x-\int_0^\infty f(t)\,\rd t\int_\Om a|\nb u(t)|^2\,\rd x\ge\left(\la_0-\|a\|_\infty\int_0^\infty f(t)\,\rd t\right)\|\nb u(t)\|_2^2\nonumber\\
& =\ell\|\nb u(t)\|_2^2.\label{eq-lower0}
\end{align}
To deal with the remaining negative term in \eqref{2.5}, we define an auxiliary energy functional
\begin{align}
\BE(t) & :=E(t)+\f1p\int_\Om k|u(t)|^p\,\rd x\nonumber\\
& =\f12\|u_t(t)\|_2^2+\f12\int_\Om A\nb u(t)\cdot\nb u(t)\,\rd x-\f12\int_0^t f(s)\,\rd s\int_\Om a|\nb u(t)|^2\,\rd x+\f12(f\circ\nb u)(t).\label{r4}
\end{align}
Then it is readily seen from \eqref{eq-lower0} that
\begin{equation}\label{eq-lower1}
\BE(t)\ge\f\ell2\|\nb u(t)\|_2^2\ge0.
\end{equation}
On the other hand, by formally differentiating \eqref{2.5} and employing the original problem \eqref{1.1}, it is not difficult to calculate
\begin{equation}\label{2.6}
E'(t)=\f12(f'\circ\nb u)(t)-\f{f(t)}2\int_\Om a|\nb u(t)|^2\,\rd x-\int_\Om b\,h(u_t(t))u_t(t)\,\rd x.
\end{equation}
Thanks to assumptions (A2)--(A4), we see that $E'(t)\le0$, that is, the energy $E(t)$ is nonincreasing.

\subsection{Some lemmas and well-posedness}

\begin{lemma}[see \cite{A3}]\label{lemma3}
Let $r$ be a constant satisfying $2\le r\le\infty$ if $n=1,2$ and $2\le r\le 2_*=\f{2n}{n-2}$ if $n\ge3$. Then there is an optimal constant $B_r>0$ depending on $r$ such that
$$
\|w\|_r^r\le B_r\|\nb w\|_2^r,\quad\forall\,w\in H_0^1(\Om).
$$
\end{lemma}

\begin{lemma}
Under assumptions {\rm(A1)--(A4),} define
\[
H(\La):=\f12\La^2-\f{K B}p\La^p,\quad\La(t):=\left\{\ell\|\nb u(t)\|_2^2+(f\circ\nb u)(t)\right\}^{1/2},
\]
where $K:=\|k\|_\infty,$ $B:=B_p\ell^{-p/2}$ and $\ell$ was the constant introduced in \eqref{2.1}. Then there holds $E(t)\ge H(\La(t))$.
\end{lemma}

\begin{proof}
A direct application of Lemma \ref{lemma3} yields
\begin{equation}
\f1p\int_\Om k|u(t)|^p\,\rd x\le\f Kp\|u(t)\|_p^p\le\f{K B_p}p\|\nb u(t)\|_2^p.\label{r3}
\end{equation}
Then we estimate
\begin{align*}
E(t) & \ge\f\ell2\|\nb u(t)\|_2^2+\f12(f\circ\nb u)(t)-\f{K B_p}p\|\nb u(t)\|_2^p\\
& \ge\f12\left\{\ell\|\nb u(t)\|_2^2+(f\circ\nb u)(t)\right\}-\f{K B}p\left\{\ell\|\nb u(t)\|_2^2+(f \circ\nb u)(t)\right\}^{p/2}\\
& =\f12\La(t)^2-\f{K B}p\La(t)^p=H(\La(t)).
\end{align*}
\end{proof}

It is easy to verify that $H(\La)$ attains a maximum at
$$
\La_1:=\left(\f1{K B}\right)^{\f1{p-2}}
$$
with the maximum
\[
E_1:=H(\La_1)=\left(\f12-\f1p\right)\La_1^2>0.
\]

\begin{lemma}[see {\cite[Lemma 3.4]{23}}]\label{lemma 2.7}
Under assumptions {\rm(A1)--(A4),} let $u$ be the solution to problem \eqref{1.1} with the initial data satisfying
\[
E(0)<E_1,\quad\La(0)<\La_1.
\]
Then there exists a constant $\La_2\in(0,\La_1)$ such that
\[
\La(t)=\left\{\ell\|\nb u(t)\|_2^2+(f\circ\nb u)(t)\right\}^{1/2}\le\La_2,\quad\forall\,t\in[0,T),
\]
where $T$ denotes the maximal existence time of the solution.
\end{lemma}

Next, we give the existence result of solutions to system \eqref{1.1}.

\begin{proposition}[Local existence]\label{theorem1}
Let $(u^0,u^1)\in H_0^1(\Om)\times L^2(\Om) $ be given and assumptions {\rm(A1)--(A4)} hold. Then there exists $T>0$ such that problem \eqref{1.1} admits a unique local weak solution $u$ on $\ov\Om\times[0,T]$ such that
$$
u\in C^1([0,T];L^2(\Om))\cap C([0,T];H_0^1(\Om)).
$$
\end{proposition}

By the Faedo-Galerkin method and the contraction mapping theorem, it is easy to prove Proposition \ref{theorem1}.  For detailed proof, please refer to \cite{H1} and \cite{Y1}.

\begin{lemma}\label{lemma2.8}
Under the same assumptions of Lemma $\ref{lemma 2.7},$ the unique weak solution $u$ to \eqref{1.1} satisfies
\begin{gather}
\f1p\int_\Om k|u(t)|^p\,\rd x\le\wt C E(t)\le\wt C E(0),\label{r1}\\
\BE(t)\le(1+\wt C)E(t)\le(1+\wt C)E(0)\label{r2},
\end{gather}
for any $t\in[0,T),$ where
$$
\wt C:=\f{2K B\La_2^{p-2}}{p-2K B\La_2^{p-2}}>0.
$$
\end{lemma}

\begin{proof}
First, by
\[
0<2K B\La_2^{p-2}<2K B\La_1^{p-2}=2<p,
\]
we confirm that the above $\wt C>0$. Then it directly follows from \eqref{r3} that
\begin{align*}
\f1p\int_\Om k|u(t)|^p\,\rd x & \le\f{K B_p}p\|\nb u(t)\|_2^p\le\f{K B}p\left(\ell^{1/2}\|\nb u(t)\|_2\right)^{p-2}\ell\|\nb u(t)\|_2^2\\
& \le\f{2K B\La_2^{p-2}}p\left(E(t)+\f1p\int_\Om k|u(t)|^p\,\rd x\right),
\end{align*}
which implies the first half of \eqref{r1}. Combining \eqref{r1} and \eqref{r4}, we obtain
\[
\BE(t)=E(t)+\f1p\int_\Om k|u(t)|^p\,\rd x\le(1+\wt C)E(t).
\]
The second halves of \eqref{r1} and \eqref{r2} are trivial due to the monotonicity of the energy $E(t)$.
\end{proof}

\begin{remark}
Clearly, \eqref{r2} indicates that $\BE(t)$ is uniformly bounded for all $t\in [0, T)$. This implies that the solution is global and the maximal existence time $T=\infty$, for detailed proof please refer to \cite{P2}. At the same time, we have $0\le E(t)\le E(0)$ for all $t\in \ov{\BR_+}$.
\end{remark}

\section{Main results and examples}\label{sec-main}

In this section, we given the energy decay rates of global solutions of problem \eqref{1.1}. The first result is as follows.

\begin{theorem}\label{theorem3.2}
Suppose that assumptions {\rm(A1)--(A4)} and the conditions in Lemma $\ref{lemma 2.7}$ hold.  Let
\begin{equation}\label{x4.1}
E(0)<\min\left\{E_1,\left(\f\ell{8K B_p}\right)^{\f2{p-2}}\f\ell{2(1+\wt C)}\right\}.
\end{equation}
Then there exists a constant $C>0$ such that for any $t\ge t_0>0,$ the energy of problem \eqref{1.1} satisfies the polynomial decay
\begin{gather}
\int_0^\infty E(t)^{\f{q+1}2}\,\rd t\le C E(0),\label{r3.1}\\
E(t)\le C E(0)(1+t)^{-\f2{q+1}}.\label{r3.2}
\end{gather}
\end{theorem}

Below, we add an extra assumption on the kernel function $f(t)$ in addition to (A4).\medskip

{\bf(A5)} Under assumption (A4), the function $f(t)$ further satisfies the following ordinary differential inequality
$$
f'(t)\le-\xi(t)G(f(t)),\quad t>0.
$$
Here $\xi\in C^1(\ov{\BR_+};\BR_+)$ is nonincreasing and $G\in C^2([0,f(0)];\ov{\BR_+})$ is a strictly increasing and strictly convex function satisfying $G(0)=G'(0)=0$.

\begin{remark}[see \cite{11}]\label{remark1}
If $G$ fulfills the above assumptions, then it allows an extension $\ov G$ in $\ov{\BR_+}$ inheriting the same monotonicity, convexity and regularity. For example, if
\[
G(f(0))=g_0,\quad G'(f(0))=g_1,\quad G''(f(0))=g_2,
\]
then we can define $\ov G$ for $t>f(0)$ by
\[
\ov G(t)=\left(g_0-g_1f(0)+\f{g_2}2f(0)^2\right)+(g_1-g_2f(0))t+\f{g_2}2t^2.
\]
With slight abuse of the notation, we will still denote such an extension by $G$. Meanwhile, the strict monotonicity of $G$ guarantees the existence of the inverse function $G^{-1}:\ov{\BR_+}\longrightarrow\ov{\BR_+}$.
\end{remark}

\begin{theorem}\label{the3.2}
Under all conditions in Theorem $\ref{theorem3.2},$ further let assumption {\rm(A5)} hold. Then there exist constants $t_0,\ve_1,k_1,k_2>0$ such that the energy of problem \eqref{1.1} satisfies
\begin{empheq}[left={E(t)\le\empheqlbrace}]{alignat=2}
& \f1{\ve_1}G_1^{-1}\left(k_1\int_{t_0}^t\,\xi(s)\,\rd s\right), & \quad & q=1,\label{7.1}\\
&E(0)\left\{t\,G_2^{-1}\left(\f{k_2}{t\int_{t_0}^t\,\xi(s)\,\rd s}\right)\right\}^{\f2{q+1}}, & \quad & q>1,\label{rr3.44}
\end{empheq}
for all $t>t_0,$ where
\[
G_1(t):=\int_t^{f(0)}\f1{s\,G'(s)}\,\rd s,\quad G_2(t):=t\,G'(\ve_1t).
\]
\end{theorem}

\begin{remark}
According to the assumptions on $G$ in (A5) and the above definitions of $G_1,G_2$, we see that $G_1$ is strictly decreasing and defines a bijection from $[0,f(0)]$ to $\ov{\BR_+}$, and $G_2$ is strictly increasing and defines a bijection from $\ov{\BR_+}$ to $\ov{\BR_+}$.
\end{remark}

\begin{remark}
Theorems \ref{theorem3.2}--\ref{the3.2} still hold for $p=2$, i.e., the polynomial nonlinearity in \eqref{1.1} degenerates to the linear one. The proofs are rather analogous and thus we omit them here.
\end{remark}

Next, we use three examples to illustrate Theorem \ref{the3.2}.

\begin{example}
Let the kernel function $f(t)$ takes the form of $f(t)=\al\,e^{-\be(1+t)}$, where $\al,\be>0$ are sufficiently small constants so that $f$ satisfies (A4). Then $f$ satisfies (A5) with $\xi(t)=\be$ and $G(t)=t$. Therefore, it follows from Theorem \ref{the3.2} and direct calculations that
$$
E(t)\le\begin{cases}
c_1\e^{-c_2t}, & q=1,\\
c_3t^{-\f2{q-1}}, & q>1,
\end{cases}
$$
where $c_1,c_2,c_3>0$ are some constants.
\end{example}

\begin{example}
Let $q=1$ and $f(t)=\al\exp(-t^\be)$, where $\be\in(0,1)$ and $\al>0$ are again small constants so that $f$ satisfies (A4). Then this time $f$ satisfies (A5) with $\xi(t)=1$ and
\[
G(t)=\f{\be\,t}{(\ln(\al/t))^{1/\be-1}}.
\]
By
$$
G'(t)=\f{(1-\be)+\be\ln(\al/t)}{(\ln(\al/t))^{1/\be}},\quad G''(t)=\f{(1-\be)(\ln(\al/t)+1/\be)}{(\ln(\al/t))^{\f1{\be+1}}},
$$
we see that $G$ also achieves the requirement in (A5) on $[0,\al]$. Then we deal with $G_1$ as
\begin{align*}
G_1(t) & =\int_t^\al\f1{s\,G'(s)}\,\rd s=\int_t^\al\f{(\ln(\al/s))^{1/\be}}{s(1-\be+\be\ln(\al/s))}\,\rd s=\int_0^{\ln(\al/t)}\f{\tau^{1/\be}}{1-\be+\be\tau}\,\rd\tau\\
& =\f1\be\int_0^{\ln(\al/t)}{\tau^{1/\be-1}}\left(\f\tau{\be^{-1}-1+\tau}\right)\rd\tau\le\f1\be\int_0^{\ln(\al/t)}\tau^{1/\be-1}\,\rd\tau=\left(\ln\f\al t\right)^{1/\be}.
\end{align*}
Then it follows from \eqref{7.1} that $E(t)\le c\exp(-c\,t^\be)$ with some constant $c>0$.
\end{example}

\begin{example}
Now let us consider $q>1$ and $f(t)=\al(1+t)^{-\be}$, where $\alpha>0$ and $\be>1$ are suitably chosen so that (A4) holds. Similarly as before, $f$ satisfies (A5) with $\xi(t)=1$ and $G(t)=t^\f{\be+1}\be$. Then again $G$ achieves the condition in (A5) hold and it follows from \eqref{rr3.44} that
$$
E(t)\le c\,t^{\min\{-\f2{q+1},\f{2(q-1)}{(q+1)^2}\}},
$$
with some constant $c>0$.
\end{example}

From the above examples, it is evident that the rates of energy decay can be significantly faster than those stated in Theorem \ref{theorem3.2} when $q=1$ and $G(t)$ is linear in the case of $q>1$. However, when $G(t)$ is nonlinear in the case of $q>1$, the rate of energy decay of the solution does not appear to increase substantially due to the weak dissipation conditions.

\section{Proof of Theorem \ref{theorem3.2}}\label{sec-proof1}

In the sequel, by $C>0$ we denote generic constants which may change from line to line. In order to prove Theorem \ref{theorem3.2}, we need to introduce some auxiliary functions and lemmas. First we define
\[
F_1(t):=\int_0^t F(t-s)\int_\Om a|\nb u(s)|^2\,\rd x\rd s,\quad F_2(t):=M(\de)(\de F_1(t)+E(t)),
\]
where $F(t),M(\de)$ were defined in \eqref{eq-def-FM}. We recall the following result.

\begin{lemma}[see {\cite[Lemma 2.1]{22}}]\label{le3.1}
For $t\ge0,$ the functions $F_1(t),F_2(t)$ defined above satisfy
\begin{align}
\f\rd{\rd t}F_1(t) & \le-\f12(f\circ\nb u)(t)+2\la_0\|\nb u(t)\|_2^2,\label{r3.3}\\
\f\rd{\rd t}F_2(t) & \le-\f12M(\de)\int_0^t K_\de(t-s)f(t-s)\int_\Om a|\nb(u(t)-u(s))|^2\,\rd x\rd s+2\la_0M(\de)\de\|\nb u(t)\|_2^2,\label{r3.4}
\end{align}
where $(f\circ\nb u)(t)$ was defined in \eqref{eq-f-grad} and $\la_0$ was the constant in \eqref{1.2}. Moreover, we have
\begin{equation}\label{rr3.5}
M(\de)\de\longrightarrow0\quad\mbox{as }\de\to0.
\end{equation}
\end{lemma}

Next, let us construct a nonnegative function $\wt a\in C^1(\ov\Om)$ satisfying $\supp\,\wt a\subset\supp\,a$ and
$$
\wt a(x)=\begin{cases}
a(x) &\mbox{if }a(x)\ge\ka/2,\\
0 &\mbox{if }a(x)\le\ka/4.
\end{cases}
$$
Then assumption (A2) gives
\begin{equation}\label{4.1}
\wt a+b\ge\f\ka2\quad\mbox{in }\Om.
\end{equation}
We invoke a useful result.

\begin{lemma}[see \cite{A3}]\label{ref1}
Let $\Om_1,\Om_2$ be subsets of $\Om$ with positive measures such that
\[
\ov\Om_1\subset\Om_2\subset\Om,\quad\mathrm{meas}(\pa\Om_2\cap\pa\Om)>0.
\]
Then there exists a constant $C_0>0$ such that
\[
\int_{\Om_1}|w|^2\,\rd x\le C_0\int_{\Om_2}|\nb w|^2\,\rd x,\quad\forall\,w\in H_0^1(\Om).
\]
\end{lemma}

According to assumption (A2), there exist a constant $\ve_0>0$ and a subset $V\subset\ov\Om$ such that
\[
\mathrm{meas}(\pa V\cap\pa\Om)>0,\quad a\ge\ve_0\mbox{ in }V.
\]
Setting
\[
\Om_1:=\supp\,\wt a,\quad\Om_2:=\{x\in\Om\mid a(x)>\min\{\ka/4,\ve_0\}\},
\]
we see that $\Om_1,\Om_2$ satisfy the assumption of Lemma \ref{ref1}. Then for any $w\in H_0^1(\Om)$, we employ Lemma \ref{ref1} to deduce
\begin{align}
\int_\Om\wt a|w|^2\,\rd x & =\int_{\Om_1}\wt a|w|^2\,\rd x\le\|\wt a\|_\infty\int_{\Om_1}|w|^2\,\rd x\le\|\wt a\|_\infty C_0\int_{\Om_2}|\nb w|^2\,\rd x\nonumber\\
& \le C\int_{\Om_2}a|\nb w|^2\,\rd x\le C\int_\Om a|\nb w|^2\,\rd x.\label{4.2}
\end{align}

Now we introduce further auxiliary functions
$$
F_3(t):=\int_\Om u_t(t)u(t)\,\rd x,\quad F_4(t):=-\int_0^t f(t-s)\int_\Om\wt a\,u_t(t)(u(t)-u(s))\,\rd x\rd s.
$$
In the next two lemmas, we provide estimates for the derivatives of the above functions.

\begin{lemma}\label{le3.2}
Given $(u^0,u^1)\in H_0^1(\Om)\times L^2(\Om)$ and under the assumptions in Lemma $\ref{lemma 2.7},$ the function $F_3(t)$ defined above satisfies the following estimate
\begin{align}
\f\rd{\rd t}F_3(t) & \le\|u_t(t)\|_2^2-\f\ell2\|\nb u(t)\|_2^2+\f{\|a\|_\infty}\ell\int_\Om a\left|\int_0^t f(t-s)\nb(u(s)-u(t))\,\rd s\right|^2\rd x\nonumber\\
& \quad\,+\f{\|b\|_\infty B_2}\ell\int_\Om b\,h(u_t(t))^2\,\rd x+\int_\Om k|u(t)|^p \,\rd x,\label{rr3.3}
\end{align}
where $\ell$ and $B_2$ were the constants introduced in \eqref{2.1} and Lemma $\ref{lemma3}$ with $r=2,$ respectively.
\end{lemma}

\begin{proof}
By the governing equation and the homogeneous boundary condition in \eqref{1.1}, we utilize the divergence theorem to calculate
\begin{align}
\f\rd{\rd t}F_3(t) & =\int_\Om\left\{u_{tt}(t)u(t)+u_t(t)^2\right\}\rd x\nonumber\\
& =\|u_t(t)\|_2^2+\int_\Om\left\{\rdiv(A\nb u(t))-\int_0^t f(t-s)\,\rdiv(a\nb u(s))\,\rd s-b\,h(u_t(t))\right\}u(t)\,\rd x\nonumber\\
& \quad\,+\int_\Om k|u(t)|^p\,\rd x\nonumber\\
& =\|u_t(t)\|_2^2+I_0(t)+I_1(t)+I_2(t)+\int_\Om k|u(t)|^p\,\rd x,\label{eq-dF3}
\end{align}
where
\begin{equation}\label{eq-est1}
I_0(t):=-\int_\Om A\nb u(t)\cdot\nb u(t)\,\rd x\le-\la_0\|\nb u(t)\|_2^2,
\end{equation}
by assumption (A1), and
\[
I_1(t):=\int_\Om a\nb u(t)\cdot\int_0^t f(t-s)\nb u(s)\,\rd s\rd x,\quad I_2(t):=-\int_\Om b\,h(u_t(t))u(t)\,\rd x.
\]

For $I_1(t)$, we separate $u(s)=u(t)+(u(s)-u(t))$ and use assumptions (A2) and (A4) to estimate
\begin{align}
I_1(t) & =\int_\Om a\nb u(t)\cdot\int_0^t f(t-s)\nb u(t)\,\rd s\rd x+\int_\Om a\nb u(t)\cdot\int_0^t f(t-s)\nb(u(s)-u(t))\,\rd s\rd x\nonumber\\
& =\int_0^t f(s)\,\rd s\int_\Om a|\nb u(t)|^2\,\rd x+\int_\Om\sqrt a\,\nb u(t)\cdot\sqrt a\int_0^t f(t-s)\nb(u(s)-u(t))\,\rd s\rd x\label{eq-est2}\\
& \le\int_0^\infty f(t)\,\rd t\int_\Om a|\nb u(t)|^2\,\rd x+\ep_1\int_\Om a|\nb u(t)|^2\,\rd x\nonumber\\
& \quad\,+\f1{4\ep_1}\int_\Om a\left|\int_0^t f(t-s)\nb(u(s)-u(t))\,\rd s\right|^2\rd x\nonumber\\
& \le\|a\|_\infty\left(\int_0^\infty f(t)\,\rd t+\ep_1\right)\|\nb u(t)\|_2^2+\f1{4\ep_1}\int_\Om a\left|\int_0^t f(t-s)\nb(u(s)-u(t))\,\rd s\right|^2\rd x.\label{eq-est3}
\end{align}
Here we applied Cauchy's inequality to treat the second term in \eqref{eq-est2}, where $\ep_1>0$ is a parameter to be determined later. In the same manner, we estimate $I_2(t)$ as
\begin{align}
I_2(t) & \le\int_\Om\sqrt b\,|h(u_t(t))|\sqrt b\,|u(t)|\,\rd x\le\ep_2\int_\Om b\,u(t)^2\,\rd x+\f1{4\ep_2}\int_\Om b\,h(u_t(t))^2\,\rd x\nonumber\\
& \le\ep_2\|b\|_\infty\|u(t)\|_2^2+\f1{4\ep_2}\int_\Om b\,h(u_t(t))^2\,\rd x\nonumber\\
& \le\ep_2\|b\|_\infty B_2\|\nb u(t)\|_2^2+\f1{4\ep_2}\int_\Om b\,h(u_t(t))^2\,\rd x\label{eq-est4}
\end{align}
again with a parameter $\ep_2>0$ to be determined later, where we utilized Lemma \ref{lemma3} with $r=2$ to deduce the last inequality.

Recalling the definition of $\ell$ in \eqref{2.1}, we substitute inequalities \eqref{eq-est1}, \eqref{eq-est3} and \eqref{eq-est4} into \eqref{eq-dF3} to derive
\begin{align*}
\f\rd{\rd t}F_3(t) & \le\|u_t(t)\|_2^2-\left|\ell-\|a\|_\infty\ep_1-\|b\|_\infty B_2\ep_2\right|\|\nb u(t)\|_2^2+\int_\Om k|u(t)|^p \,\rd x\\
& \quad\,+\f1{4\ep_1}\int_\Om a\left|\int_0^t f(t-s)\nb(u(s)-u(t))\,\rd s\right|^2\rd x+\f1{4\ep_2}\int_\Om b\,h(u_t(t))^2\,\rd x
\end{align*}
Consequently, we arrive at \eqref{rr3.3} by simply choosing
\[
\ep_1=\f\ell{4\|a\|_\infty},\quad\ep_2=\f\ell{4\|b\|_\infty B_2},
\]
which completes the proof.
\end{proof}

\begin{lemma}\label{le3.3}
Given $(u^0,u^1)\in H_0^1(\Om)\times L^2(\Om)$ and under the assumptions in Lemma $\ref{lemma 2.7},$ the function $F_4(t)$ defined above satisfies the following estimate
\begin{align}
\f\rd{\rd t}F_4(t) & \le\ve\|u_t(t)\|_2^2+\ve(1+B_{2(p-1)}\eta_1)\|\nb u(t)\|_2^2+C(\ve)\int_\Om a\left|\int_0^t f(t-s)\nb(u(t)-u(s))\,\rd s\right|^2\rd x\nonumber\\
& \quad\,-\int_0^t f(s)\,\rd s\int_\Om\wt a\,u_t(t)^2\,\rd x+C(\ve)\int_\Om b\,h^2(u_t(t))\,\rd x-\f{f(0)C}{4\ve}(f'\circ\nb u)(t).\label{xx3.1}
\end{align}
Here $\ve\in(0,1)$ is an arbitrary constant, $C(\ve)>0$ is a constant depending on $\ve$ and
$$
\eta_1:=\left(\f{2(1+\wt C)E(0)}\ell\right)^{p-2},
$$
where $\wt C$ was the constant introduced in Lemma $\ref{lemma2.8}$.
\end{lemma}

\begin{proof}
Similarly to the proof of Lemma \ref{le3.2}, we differentiate $F_4(t)$ by its definition and exploit problem \eqref{1.1} to calculate
\begin{align*}
\f\rd{\rd t}F_4(t) & =-\int_\Om u_{tt}(t)\wt a\int_0^t f(t-s)(u(t)-u(s))\,\rd s\rd x\\
& \quad\,-\int_\Om\wt a\,u_t(t)\int_0^t f'(t-s)(u(t)-u(s))\,\rd s\rd x -\int_0^t f(s)\,\rd s\int_\Om\wt a\,u_t(t)^2\,\rd x\\
& =\int_\Om\left\{-\rdiv(A\nb u(t))+\int_0^t f(t-s)\,\rdiv(a\nb u(s))\,\rd s+b\,h(u_t(t))-k\,|u(t)|^{p-2}u(t)\right\}\\
& \qquad\:\:\times\wt a\int_0^t f(t-s)(u(t)-u(s))\,\rd s\rd x+I_9(t)+I_{10}(t)\\
& =\sum_{i=3}^{10}I_i(t),
\end{align*}
where
\begin{align*}
I_3(t) & :=\int_\Om\wt a\,A\nb u(t)\cdot\int_0^t f(t-s)\nb(u(t)-u(s))\,\rd s\rd x,\\
I_4(t) & :=\int_\Om\nb\wt a\cdot A\nb u(t)\int_0^t f(t-s)(u(t)- u(s))\,\rd s\rd x,\\
I_5(t) & :=-\int_\Om\wt a\int_0^t f(t-s)a\nb u(s)\,\rd s\cdot\int_0^t f(t-s)\nb(u(t)-u(s))\,\rd s\rd x,\\
I_6(t) & :=-\int_\Om\nb\wt a\cdot\int_0^t f(t-s)a\nb u(s)\,\rd s\int_0^t f(t-s)(u(t)-u(s))\,\rd s\rd x,\\
I_7(t) & :=\int_\Om b\,h(u_t(t))\left(\int_0^t f(t-s)\wt a(u(t)-u(s))\,\rd s\right)\rd x,\\
I_8(t) & :=-\int_\Om\wt a\,k|u(t)|^{p-2}u(t)\left(\int_0^t f(t-s)(u(t)-u(s))\,\rd s\right)\rd x,\\
I_9(t) & :=-\int_\Om\wt a\,u_t(t)\int_0^t f'(t-s)(u(t)-u(s))\,\rd s\rd x,\quad
I_{10}(t):=-\int_0^t f(s)\,\rd s\int_\Om\wt a\,u_t(t)^2\,\rd x.
\end{align*}
We notice that $I_{10}(t)$ already appears in the desired estimate \eqref{xx3.1}. Then we shall estimate the remaining terms above one by one.

Using Young's inequality and \eqref{4.2}, one has
\begin{align}
|I_3(t)| & \le\int_\Om\left|\sqrt{\wt a}\,\nb u(t)\cdot\sqrt{\wt a}\,A\int_0^t f(t-s)\nb(u(t)-u(s))\,\rd s\right|\rd x\nonumber\\
& \le\f\ve{4\|\wt a\|_\infty}\int_\Om\wt a\,|\nb u(t)|^2\,\rd x+\f{\|\wt a\|_\infty}\ve\int_\Om\wt a\left|A\int_0^t f(t-s)\nb(u(t)-u(s))\,\rd s\right|^2\rd x\nonumber\\
& \le\f\ve4\|\nb u(t)\|_2^2+\f{\|\wt a\|_\infty\wt\la}\ve\int_\Om a\left|\int_0^t f(t-s)\nb(u(t)-u(s))\,\rd s\right|^2\rd x,\label{est-I3}
\end{align}
where $\wt\la:=\max_{x\in\ov\Om}\|A(x)\|_{\ell^2}^2$ and $\|A(x)\|_{\ell^2}$ denotes the matrix $2$-norm of $A(x)\in\BR^{n\times n}$, i.e., $|A(x)\xi|\le\|A(x)\|_{\ell^2}|\xi|$ for any $\xi\in\BR^n$. In the same manner and employing the definition of $\wt a$ and \eqref{4.2}, we deduce that
\begin{align*}
|I_4(t)| & \le\f\ve4\|\nb u(t)\|_2^2+\f1\ve\int_\Om \left|\nb\wt a\,A\int_0^t f(t-s)(u(t)-u(s))\,\rd s\right|^2\rd x\\
& \le\f\ve4\|\nb u(t)\|_2^2+\f{\|\nb a\|_\infty^2\wt\la}\ve\int_{\Om_1}\left|\int_0^t f(t-s)(u(t)-u(s))\,\rd s\right|^2\rd x\\
& \le\f\ve4\|\nb u(t)\|_2^2+\f{C\|\nb a\|_\infty^2\wt\la}\ve\int_{\Om_2}a\left|\int_0^t f(t-s)\nb(u(t)-u(s))\,\rd s\right|^2\rd x\\
& \le\f\ve4\|\nb u(t)\|_2^2+\f{C\|\nb a\|_\infty^2\wt\la}\ve\int_\Om a\left|\int_0^t f(t-s)\nb(u(t)-u(s))\,\rd s\right|^2\rd x.
\end{align*}
For $I_5(t)$, we separate $u(s)=u(t)+(u(s)-u(t))$ and use assumptions (A2) and (A4) to estimate
\begin{align}
|I_5(t)| & =\left|\int_\Om\wt a\int_0^t f(t-s)a\nb(u(s)-u(t)+u(t))\,\rd s\cdot\int_0^t f(t-s)\nb(u(t)-u(s))\,\rd s\rd x\right|\nonumber\\
& \le\|\wt a\|_\infty\int_\Om a\left|\int_0^t f(t-s)\nb(u(t)-u(s))\,\rd s\right|^2\rd x\nonumber\\
& \quad\,+\|\wt a\|_\infty\int_0^t f(s)\,\rd s\int_\Om a\nb u(t)\cdot\int_0^t f(t-s)\nb(u(t)-u(s))\,\rd s\rd x\nonumber\\
& \le\f\ve 4\|\nb u(t)\|_2^2+\|\wt a\|_\infty\left(1+\f{(\la_0-\ell)^2}{\ve\|a\|_\infty}\right)\int_\Om a\left|\int_0^t f(t-s)\nb(u(t)-u(s))\,\rd s\right|^2\rd x\nonumber\\
& \le\f\ve 4\|\nb u(t)\|_2^2+C(\ve)\int_\Om a\left|\int_0^t f(t-s)\nb(u(t)-u(s))\,\rd s\right|^2\rd x.
\end{align}
Similarly, we estimate $I_6(t)$ as
\begin{equation}
|I_6(t)|\le\f\ve4\|\nb u(t)\|_2^2+C(\ve)\int_\Om a\left|\int_0^t f(t-s)\nb(u(t)-u(s))\,\rd s\right|^2\rd x.
\end{equation}
In the same manner, we estimate $I_7(t)$ and $I_9(t)$ as
\begin{align}
|I_7(t)| & \le\ve\int_\Om a\left|\int_0^t f(t-s)\nb(u(t)-u(s))\,\rd s\right|^2\rd x+C(\ve)\int_\Om b\,h(u_t(t))^2\,\rd x,\\
|I_9(t)| & \le\ve\|u_t(t)\|_2^2+\f1{4\ve}\int_\Om\left|\wt a\int_0^t f'(t-s)\nb(u(t)-u(s))\,\rd s\right|^2\rd x\nonumber\\
& \le\ve\|u_t(t)\|_2^2+\f C{4\ve}\int_\Om a\int_0^t f'(t-s)\,\rd s\int_0^t f'(t-s)|\nb(u(t)-u(s))|^2\,\rd s\rd x\nonumber\\
& \le\ve\|u_t(t)\|_2^2-\f{f(0)C}{4\ve}(f'\circ\nb u)(t).
\end{align}
Finally, for $I_8(t)$, we utilize Young's inequality and \eqref{4.2} to derive
\begin{align}
|I_8(t)| & \le\ve\int_\Om|u(t)|^{2(p-1)}\,\rd x+\f1{4\ve}\int_\Om k^2\,\wt a^2\left|\int_0^t f(t-s)(u(t)-u(s))\,\rd s\right|^2\rd x\nonumber\\
& \le\ve\|u(t)\|_{2(p-1)}^{2(p-1)}+\f {K^2C\|\wt a\|_\infty}{4\ve}\int_\Om a\left|\int_0^t f(t-s)\nb(u(t)-u(s))\,\rd s\right|^2\rd x\nonumber\\
& \le\ve B_{2(p-1)}\|\nb u(t)\|_2^{2(p-1)}+C(\ve)\int_\Om a\left|\int_0^t f(t-s)\nb(u(t)-u(s))\,\rd s\right|^2\rd x\nonumber\\
& \le\ve B_{2(p-1)}\eta_1\|\nb u(t)\|_2^2+C(\ve)\int_\Om a\left|\int_0^t f(t-s)\nb(u(t)-u(s))\,\rd s\right|^2\rd x,\label{est-I8}
\end{align}
where we used Lemma \ref{lemma3} with $r=2(p-1)$ and recall $K=\|k\|_\infty$. For the first term in \eqref{est-I8}, we combine \eqref{r2} with \eqref{eq-lower1} to deduce
\[
\|\nb u(t)\|_2^{2(p-2)}\le\left(\f{2\,\BE(t)}\ell\right)^{p-2}\le\left(\f{2(1+\wt C)E(0)}\ell\right)^{p-2}=\eta_1.
\]
Eventually, we arrive at \eqref{xx3.1} by combining all above estimates \eqref{est-I3}--\eqref{est-I8}.
\end{proof}

We are now prepared to demonstrate Theorem \ref{theorem3.2}. First we define an auxiliary function
$$
L(t):=N_1E(t)+F_3(t)+N_2F_4(t),
$$
where $N_1,N_2>0$ are constants to be determined soon. Using Young's inequality and \eqref{4.2}, we easily obtain
\begin{align*}
|F_3(t)| & \le\f{B_2}2\|\nb u(t)\|_2^2+\f12\|u_t(t)\|_2^2,\\
|F_4(t)| & \le\f12\|u_t(t)\|_2^2+\f{\|\wt a\|_\infty}2\int_\Om\wt a\left|\int_0^t f(t-s)(u(t)-u(s))\,\rd s\right|^2\rd x\\
& \le\f12\|u_t(t)\|_2^2+\f{C\|\wt a\|_\infty}2\int_\Om a\int_0^t f(t-s)\,\rd s\int_0^t f(t-s)|\nb(u(t)-u(s))|^2\,\rd s\rd x\\
& \le\f12\|u_t(t)\|_2^2+\f{C\la_0}2\int_\Om a\int_0^t f(t-s)|\nb(u(t)-u(s))|^2\,\rd s\rd x\\
& =\f12\|u_t(t)\|_2^2+\f{C\la_0}2(f\circ\nb u)(t).
\end{align*}
Then by choosing $N_1$ and $N_2$ suitably, we have the equivalence $L(t)\sim E(t)$.

Fix any $t_0>0$ and put $f_1:=\int_0^{t_0}f(s)\,\rd s$. Combining \eqref{2.6}, Lemma \ref{le3.2} and Lemma \ref{le3.3} with
\[
\ve=\f\ell{4N_2(1+B_{2(p-1)}\eta_1)},
\]
we obtain for all $t\ge t_0$ that
\begin{align*}
L'(t) & \le-\f\ell4\|\nb u(t)\|_2^2+(\ve N_2+1)\|u_t(t)\|_2^2-N_2f_1\int_\Om(\wt a+b)u_t(t)^2\,\rd x
+\int_\Om k|u(t)|^p\,\rd x\\
& \quad\,+(C(\ve)N_2+C)\int_\Om a\left|\int_0^t f(t-s)\nb(u(t)-u(s)\,\rd s\right|^2\rd x\\
& \quad\,+\left(\f{N_1}2-\f{f(0)C N_2}{4\ve}\right)(f'\circ\nb u)(t)+(C(\ve)N_2+C)\int_\Om b(h(u_t(t))^2+u_t(t)^2)\,\rd x.
\end{align*}
At this point, we first use \eqref{4.1} and choose $N_2$ sufficiently large so that
$$
\te:=N_2f_1\f\ka2-\f\ell{4(1+B_{2(p-1)}\eta_1)}-1>0,
$$
and then choose $N_1$ sufficiently large so that
$$
\f{N_1}2-\f{f(0)CN_2}{4\ve}>0.
$$
Hence we arrive at
\begin{align}
L'(t) & \le-\f\ell4\|\nb u(t)\|_2^2-\te\int_\Om u_t(t)^2\,\rd x+\int_\Om k|u(t)|^p\,\rd x+(C(\ve)N_2+C)\int_\Om b(h(u_t(t))^2+u_t(t)^2)\,\rd x\nonumber\\
& \quad\,+(C(\ve)N_2+C)\int_\Om a\left|\int_0^t f(t-s)\nb(u(t)-u(s)\,\rd s\right|^2\rd x.\label{4.5}
\end{align}
For the last term above, we recall the definition \eqref{eq-def-FM} of $K_\de(t)$ and $M(\de)$ to estimate
\begin{align}
& \quad\,\int_\Om a\left|\int_0^t f(t-s)\nb(u(t)-u(s)\,\rd s\right|^2\rd x\nonumber\\
&=\int_\Om a\left|\int_0^t \sqrt{\f{f(t-s)}{K_\de(t-s)}}\sqrt{f(t-s)K_\de(t-s)}\nb(u(t)-u(s)\,\rd s\right|^2\rd x\nonumber\\
& \le\int_0^t\f{f(s)}{K_\de(s)}\,\rd s\int_\Om a\int_0^t K_\de(t-s)f(t-s)|\nb(u(t)-u(s))|^2\,\rd s\rd x\nonumber\\
& \le M(\de)\int_0^t K_\de(t-s)f(t-s)\int_\Om a|\nb(u(t)-u(s))|^2\,\rd x\rd s.\label{4.6}
\end{align}

Next, we further define
$$
J(t):=L(t)+\f\ell{32\la_0}F_1(t)+{2(C(\ve)N_2+C)}F_2(t).
$$
Then we apply inequalities \eqref{r3.3}--\eqref{r3.4} in Lemma \ref{le3.1} and \eqref{4.6} to \eqref{4.5} to derive
\begin{align*}
J'(t) & \le-\left(\f\ell4-\f\ell{16}-{4(C(\ve)N_2+C)\la_0\de M(\de)}\right)\|\nb u(t)\|_2^2-\te\|u_t(t)\|_2^2-\f\ell{64\la_0 }(f\circ\nb u)(t)\\
& \quad\,+\int_\Om k|u(t)|^p\,\rd x+(C(\ve)N_2+C)\int_\Om b(h(u_t(t))^2+u_t(t)^2)\,\rd x.
\end{align*}
Here, the convergence \eqref{rr3.5} guarantees the existence of a constant $\de_0>0$ such that
$$
4(C(\ve)N_2+C)\la_0\de M(\de)<\f\ell{16},
$$
for any $0<\de<\de_0$. Recalling the original energy $E(t)$, we deduce
\begin{align*}
J'(t) & \le-\f\ell8\|\nb u(t)\|_2^2-\te\|u_t(t)\|_2^2-\f\ell{64\la_0 }(f\circ\nb u)(t)+\int_\Om k|u(t)|^p\,\rd x\\
& \quad\,+(C(\ve)N_2+C)\int_\Om b(h(u_t(t))^2+u_t(t)^2)\,\rd x\\
& \le-\ep\,E(t)-\left(\f\ell8-\f{\ep\mu_0}2\right)\|\nb u(t)\|_2^2-\left(\te-\f\ep2\right)\|u_t(t)\|_2^2-\left(\f\ell{64\la_0}-\f\ep2\right)(f\circ\nb u)(t)\\
& \quad\,+\left(1-\f\ep p\right)\int_\Om k|u(t)|^p\,\rd x+(C(\ve)N_2+C)\int_\Om b(h(u_t(t))^2+u_t(t)^2)\,\rd x,
\end{align*}
where $\ep>0$ is another adjustable parameter. Utilizing \eqref{eq-lower1}, \eqref{r3} and \eqref{r2}, we bound
\begin{align}
\int_\Om k|u(t)|^p\,\rd x & \le K B_p\|\nb u(t)\|_2^p\le K B_p\left(\f{2\,\BE(t)}\ell\right)^{\f{p-2}2}\|\nb u(t)\|_2^2\nonumber\\
& \le K B_p\left(\f{2(1+\wt C)E(0)}\ell\right)^{\f{p-2}2}\|\nb u(t)\|_2^2.\label{4.9}
\end{align}
Thanks to the key assumption \eqref{x4.1}, it is possible to choose a sufficiently small $\ep>0$ such that
$$
\f\ell8-\f{\ep\mu_0}2-K B_p\left(\f{2(1+\wt C)E(0)}\ell\right)^{\f{p-2}2}>0,\quad\te-\f\ep2>0,\quad\f\ell{64\la_0}-\f\ep2>0,
$$
and hence
\begin{equation}\label{4.10}
J'(t)+\ep\,E(t)\le C\int_\Om b(h(u_t(t))^2+u_t(t)^2)\,\rd x.
\end{equation}
In the sequel, we discuss the cases of $q=1$ and $q>1$ separately.\medskip

{\bf Case 1.} In the case of $q=1$, the assumption (A3) asserts that $h$ behaves linearly in $\BR$. Then it follows from \eqref{4.10} that
\begin{equation}\label{4.11}
J'(t)+\ep\,E(t)\le C\int_\Om b\,h(u_t(t))u_t(t)\,\rd x\le-C E'(t),
\end{equation}
where we employed \eqref{2.6} to obtain the second inequality above. Picking any $\tau>t_0>0$, we integrate the above inequality over $(t_0,\tau)$ to estimate
\[
J(\tau)-J(t_0)+\ep\int_{t_0}^\tau E(t)\,\rd t\le-C\int_{t_0}^\tau E'(t)\,\rd t=C(E(t_0)-E(\tau))\le C E(0),
\]
due to the monotonicity of $E(t)$. Since $J(\tau)\ge0$ and $J(t_0)\le C E(0)$ for some $t_0>0$, we conclude
\[
\int_{t_0}^\tau E(t)\,\rd t\le C J(t_0)+C E(0)\le C E(0).
\]
Since $\tau>t_0$ was chosen arbitrarily, the above estimate holds also as $\tau\to\infty$. Hence we arrive at \eqref{r3.1} by
\[
\int_0^\infty E(t)\,\rd t=\left(\int_0^{t_0}+\int_{t_0}^\infty\right)E(t)\,\rd t\le t_0E(0)+C E(0)\le C E(0),
\]
again by the monotonicity of $E(t)$, which further indicates \eqref{r3.2} immediately.\medskip

{\bf Case 2.} In the case of $q>1$, we apply \cite[Equation (4.11)]{22} to the right-hand side of \eqref{4.10} to derive
\begin{equation}\label{4.14}
J'(t)+\ep\,E(t)\le C\left\{-E'(t)+(-E'(t))^{\f2{q+1}}\right\}.
\end{equation}
Now introduce an auxiliary function $\vp:\ov{\BR_+}\longrightarrow\ov{\BR_+}$ which is nonincreasing and locally absolutely continuous. For the constant $t_0>0$ in Case 1 and any $\tau>t_0$, we multiply \eqref{4.14} by $\vp(t)$ and integrate over $(t_0,\tau)$ to deduce
\begin{equation}\label{eq-est-q0}
\int_{t_0}^\tau\vp(t)E(t)\,\rd t\le-C\int_{t_0}^\tau\vp(t)(J(t)+C E(t))'\,\rd t+C\int_{t_0}^\tau\vp(t)(-E'(t))^{\f2{q+1}}\,\rd t.
\end{equation}
For the first term on the right-hand side above, we perform integration by parts to calculate
\begin{align*}
-\int_{t_0}^\tau\vp(t)(J(t)+C E(t))'\,\rd t & =\Big[\vp(t)(J(t)+C E(t))\Big]_\tau^{t_0}+\int_{t_0}^\tau\vp'(t)(J(t)+C E(t))\,\rd t\\
& \le\vp(t_0)(J(t_0)+C E(t_0))\le C\vp(0)E(0),
\end{align*}
by the monotonicity of $\vp(t),E(t)$ and the same argument as before. For the second term on the right-hand side of \eqref{eq-est-q0}, we take advantage of Young’s inequality for products to estimate
\[
C\int_{t_0}^\tau\vp(t)(-E'(t))^{\f2{q+1}}\,\rd t\le\f12\int_{t_0}^\tau\vp(t)^{\f{q+1}{q-1}}\,\rd t-C\int_{t_0}^\tau E'(t)\,\rd t\le C E(0)+\f12\int_{t_0}^\tau\vp(t)^{\f{q+1}{q-1}}\,\rd t.
\]
Especially, taking $\vp(t)=E(t)^{\f{q-1}2}$ gives $\vp(t)E(t)=\vp(t)^{\f{q+1}{q-1}}=E(t)^{\f{q+1}2}$. Thus we substitute the above two inequalities into \eqref{eq-est-q0} to conclude
$$
\int_{t_0}^\tau E(t)^{\f{q+1}2}\,\rd t\le CE(0).
$$
The remaining proof is identical to that of Case 1.

\section{Proof of Theorem \ref{the3.2}}\label{sec-proof2}

To establish Theorem \ref{the3.2}, we require the following lemma.

\begin{lemma}\label{lem3.6}
Under assumptions {\rm(A4)--(A5),} there exists a sufficiently small $\ga\in(0,1)$ such that the term $(f\circ\nb u)(t)$ defined by \eqref{eq-f-grad} satisfies
\begin{empheq}[left={(f\circ\nb u)(t)\le\empheqlbrace}]{alignat=2}
& \f1\ga G^{-1}\left(\f{\ga\,\mu(t)}{\xi(t)}\right), & \quad & q=1,\label{3.28}\\
& \f t\ga G^{-1}\left(\f{\ga\,\mu(t)}{t\,\xi(t)}\right), & \quad & q>1,\label{r3.28}
\end{empheq}
where $\xi,G$ were introduced in assumption {\rm(A5)} and
\begin{equation}\label{x3.28}
\mu(t):=-(f'\circ\nb u)(t)\le-2E'(t).
\end{equation}
\end{lemma}

\begin{proof}
Recall that the existence of the inverse function $G^{-1}$ was mentioned in Remark \ref{remark1}. Meanwhile, the inequality in \eqref{x3.28} follows immediately from \eqref{2.6} and assumptions (A2)--(A4).

For $q=1$, let us fix any $t>0$ and set
$$
\psi(s):=\int_\Om a|\nb(u(t)-u(s))|^2\,\rd x.
$$
Taking advantage of \eqref{eq-lower1}, \eqref{r2} and \eqref{r3.1} in Theorem \ref{theorem3.2}, we invoke the nonnegativity and the monotonicity of $E(t)$ to dominate
\begin{align*}
\int_0^t\psi(s)\,\rd s & \le2\|a\|_\infty\int_0^t\left\{\|\nb u(t)\|_2^2+\|\nb u(s)\|_2^2\right\}\rd s\le\f{4\|a\|_\infty}\ell\int_0^t(\BE(t)+\BE(s))\,\rd s\\
& \le C\int_0^t(E(t)+E(s))\,\rd s\le C\int_0^t E(s)\,\rd s\le C\int_0^\infty E(t)\,\rd t\le C E(0)<\infty.
\end{align*}
Then we can choose a sufficiently small  $\ga>0$ such that $\int_0^t\ga\,\psi(s)\,\rd s=1$, that is, $\ga\,\psi(s)\,\rd s$ is a probability measure on $(0,t)$. By assumption (A5) and especially the convexity of $G$, we utilize Jensen's inequality to deduce
\begin{align*}
\mu(t) & =\f1\ga\int_0^t(-f'(t-s))\left(\ga\int_\Om a|\nb(u(t)-u(s))|^2\,\rd x\right)\rd s\ge\f1\ga\int_0^t\,\xi(t-s)G(f(t-s))\,\ga\,\psi(s)\,\rd s\\
& \ge\f{\xi(t)}\ga\int_0^t G(f(t-s))\,\ga\,\psi(s)\,\rd s\ge\f{\xi(t)}\ga G\left(\int_0^t f(t-s)\,\ga\,\psi(s)\,\rd s\right)=\f{\xi(t)}\ga G(\ga(f\circ\nb u)(t)).
\end{align*}
which indicates \eqref{3.28} by performing $G^{-1}$ on both sides above.

Next we deal with $q>1$. In the same manner, we obtain
\begin{align*}
\f1t\int_0^t\psi(s)\,\rd s & \le\f{2\|a\|_\infty}t\int_0^t\left\{\|\nb u(t)\|_2^2+\|\nb u(s)\|_2^2\right\}\rd s\le\f{4\|a\|_\infty}{\ell\,t}\int_0^t(\BE(t)+\BE(s))\,\rd s\\
& \le \f Ct\int_0^tE(s)\,\rd s \le \f Ct\int_0^tE(0)\,\rd s\le C E(0)<\infty.
\end{align*}
Then again we can choose a sufficiently small $\ga>0$ such that $\f\ga t\int_0^t\psi(s)\,\rd s=1$. Parallel to the above argument, it follows that
\begin{align*}
\mu(t) & =\f t\ga\int_0^t(-f'(t-s))\left(\f\ga t\int_\Om a|\nb(u(t)-u(s))|^2\,\rd x\right)\rd s\\
& \ge\f t\ga\int_0^t\,\xi(t-s)G(f(t-s))\,\f\ga t\,\psi(s)\,\rd s\ge\f{\xi(t)t}\ga\int_0^t G(f(t-s))\,\f\ga t\,\psi(s)\,\rd s\\
& \ge\f{\xi(t)t}\ga G\left(\f\ga t\,\int_0^t f(t-s)\,\psi(s)\,\rd s\right)=\f{\xi(t)t}\ga G\left(\f\ga t(f\circ\nb u)(t)\right),
\end{align*}
 which indicates \eqref{r3.28}.
\end{proof}

Now we are in a position to prove Theorem \ref{the3.2}. As before, we divide the cases of $q=1$ and $q>1$ separately.\medskip

{\bf Case 1.} In the case of $q=1$, first we use H\"older's inequality to deal with
\begin{align}
\int_\Om a\left|\int_0^t f(t-s)\nb(u(t)-u(s))\,\rd s\right|^2\rd x & \le\int_\Om a\int_0^t f(s)\,\rd s\int_0^t f(t-s)|\nb(u(t)-u(s))|^2\,\rd x\rd s\nonumber\\
& \le\f{\la_0-\ell}{\|a\|_\infty}(f\circ\nb u)(t).\label{x6.1}
\end{align}
Combining \eqref{4.5}, \eqref{4.9}, \eqref{4.11} with \eqref{x6.1},  we derive the necessary estimate as follows
\begin{align*}
L'(t)\le-\ep\,E(t)+C(f\circ\nb u)(t)-C E'(t).
\end{align*}
Using the estimate \eqref{3.28} in Lemma \ref{lem3.6} and putting $L_1(t):=L(t)+C E(t)\sim E(t)$, we get
\[
L_1'(t)\le-\ep\,E(t)+ C G^{-1}\left(\f{\ga\,\mu(t)}{\xi(t)}\right).
\]
Now we pick a constant $\ve_1\in(0,f(0))$ and define
\[
S(t):=G'\left(\ve_1\f{E(t)}{E(0)}\right)L_1(t),
\]
which satisfies $S(t)\sim E(t)$. Noting that $G''(t)\ge0$ and $E'(t)\le 0$, we estimate
\begin{align}
S'(t) & =G'\left(\ve_1\f{E(t)}{E(0)}\right)L_1'(t)+\ve_1\f{E'(t)}{E(0)}G''\left(\ve_1\f{E(t)}{E(0)}\right)L_1(t)\nonumber\\
& \le-\ep\,E(t)G'\left(\ve_1\f{E(t)}{E(0)}\right)+C G'\left(\ve_1\f{E(t)}{E(0)}\right)G^{-1}\left(\f{\ga\,\mu(t)}{\xi(t)}\right).\label{x6.5}
\end{align}
Next, let $G^*$ be the convex conjugate of $G$ in the sense of Young (see \cite{5}), that is,
\[
G^*(s):=s\,(G')^{-1}(s)-G((G')^{-1}(s)),\quad s\in[0,G'(f(0))].
\]
On the other hand, it is known that $G^*$ satisfies the generalized Young's inequality
\begin{equation}\label{x6.7}
A B\le G^*(A)+G(B),\quad A\in[0,G'(f(0))],\ B\in[0,f(0)].
\end{equation}	
Taking $A=G'(\ve_1\f{E(t)}{E(0)})$ and $B=G^{-1}(\f{\ga\,\mu(t)}{\xi(t)})$ in \eqref{x6.7}, we can estimate \eqref{x6.5} as
\begin{align*}
S'(t)&\le-\ep\,E(t)G'\left(\ve_1\f{E(t)}{E(0)}\right)+C G^*\left(G'\left(\ve_1\f{E(t)}{E(0)}\right)\right)+C\f{\ga\,\mu(t)}{\xi(t)}\\
&\le-\ep\,E(t)G'\left(\ve_1\f{E(t)}{E(0)}\right)+C\ve_1\f{E(t)}{E(0)}G'\left(\ve_1\f{E(t)}{E(0)}\right)+C\f{\ga\,\mu(t)}{\xi(t)}.
\end{align*}
Multiplying both sides of the above estimate by $\xi(t)$, we use \eqref{x3.28} and the facts that $\ve_1\f{E(t)}{E(0)}<f(0)$ to derive
\begin{align*}
\xi(t)S'(t) & \le-\ep\,\xi(t)E(t)G'\left(\ve_1\f{E(t)}{E(0)}\right)+C\xi(t)\ve_1\f{E(t)}{E(0)}G'\left(\ve_1\f{E(t)}{E(0)}\right)+C\ga\,\mu(t)\\
& \le-(\ep\,E(0)-C\ve_1)\xi(t)\f{E(t)}{E(0)}G'\left(\ve_1\f{E(t)}{E(0)}\right)-C E'(t).
\end{align*}

Further setting $S_1(t):=\xi(t)S(t)+C E(t)$, we know $S_1(t)\sim E(t)$, that is, there exist constants $\al_1,\al_2>0$ such that
\begin{equation}\label{xx6.10}
\al_1S_1(t)\le E(t)\le\al_2S_1(t).
\end{equation}
With a suitable choice of $\ve_1>0$, we have
\begin{equation}\label{xx6.11}
S_1'(t)\le-\al_3\xi(t)\f{E(t)}{E(0)}G'\left(\ve_1\f{E(t)}{E(0)}\right)=-\al_3\xi(t)G_2\left(\f{E(t)}{E(0)}\right),
\end{equation}
where $\al_3:=\ep\,E(0)-C\ve_1>0$ and $G_2(t)=t\,G'(\ve_1t)$. By
\[
G_2'(t)=G'(\ve_1t)+\ve_1t\,G''(\ve_1t)
\]
and the properties of $G$, we see that $G_2,G_2'>0$ on $[0,f(0)]$. Finally introducing $S_2(t):=\f{\al_1S_1(t)}{E(0)}\sim E(t)$, we employ \eqref{xx6.10}--\eqref{xx6.11} to obtain
\[
S_2'(t)\le-k_1{\xi(t)}G_2\left(\f{E(t)}{E(0)}\right)\le-k_1\xi(t)G_2\left(\f{\al_1S_1(t)}{E(0)}\right)=-k_1\xi(t)G_2(S_2(t)),
\]
where $k_1:=\f{\al_1\al_3}{E(0)}$. Therefore, we divide both sides of the above inequality by $G_2(S_2(t))$ and integrate over some interval $(t_0,t)\subset\BR_+$ to deduce
\[
k_1\int_{t_0}^t\,\xi(s)\,\rd s\le-\int_{t_0}^t\f{S'_2(s)}{G_2(S_2(s))}\,\rd s=\int_{\ve_1S_2(t)}^{\ve_1S_2(t_0)}\f1{s\,G'(s)}\,\rd s
\]
by some appropriate change of variables. Fixing $t_0$ such that $\ve_1S_2(t_0)=f(0)$ and recalling the definition of $G_1(t)$ in Theorem \ref{the3.2}, we can rewrite the above inequality as
\[
G_1(\ve_1S_2(t))\ge k_1\int_{t_0}^t\,\xi(s)\,\rd s.
\]
Owing to the fact that $G_1$ is strictly decreasing on $[0,f(0)]$, we conclude
\[
S_2(t)\le\f1{\ve_1}G_1^{-1}\left(k_1\int_{t_0}^t\,\xi(s)\,\rd s\right),
\]
which eventually implies \eqref{7.1} by the equivalence of $S_2(t)$ and $E(t)$.\medskip

{\bf Case 2.} Now let us treat the case of $q>1$. First we combine the estimate \eqref{r3.28} in Lemma \ref{lem3.6} with \eqref{4.14} and \eqref{x6.1} to obtain
\[
L'(t)\le-\ep\,E(t)+ C\f t\ga G^{-1}\left(\f{\ga\,\mu(t)}{t\,\xi(t)}\right)-C E'(t)+C(-E'(t))^{\f2{q+1}}.
\]
In an analogous manner as before, we multiply both sides of the above inequality by $E(t)^{\f{q-1}2}$ and employ Young's inequality for products to estimate
\begin{align*}
E(t)^{\f{q-1}2}L'(t) & \le-\ep\,E(t)^{\f{q+1}2}+C E(0)^{\f{q-1}2}t\,G^{-1}\left(\f{\ga\,\mu(t)}{t\,\xi(t)}\right)-C E'(t)E(t)^{\f{q-1}2}+C(-E'(t))^{\f2{q+1}}E(t)^{\f{q-1}2}\\
& \le-\f\ep2E(t)^{\f{q+1}2}+ C\,t\,G^{-1}\left(\f{\ga\,\mu(t)}{t\,\xi(t)}\right)-C E'(t)E(t)^{\f{q-1}2}-C E'(t).
\end{align*}
Introducing an auxiliary function
\[
R(t):=E(t)^{\f{q-1}2}L(t)+C E(t)^{\f{q+1}2}+C E(t),
\]
we have $R(t)\sim E(t)$ and
\begin{equation}\label{eq-est-R0}
R'(t)\le-\f\ep2E(t)^{\f{q+1}2}+C\,t\,G^{-1}\left(\f{\ga\,\mu(t)}{t\,\xi(t)}\right).
\end{equation}

Next, we pick a constant $\ve_1\in(0,f(0))$ and define
\[
R_1(t):=G'\left(\f{\ve_1}t\left(\f{E(t)}{E(0)}\right)^{\f{q+1}2}\right)R(t),
\]
which satisfies $R_1(t)\sim E(t)$. Noticing $G''(t)\ge0$ and $E'(t)\le 0$, we use \eqref{eq-est-R0} to deduce
\begin{align*}
R_1'(t) & \le G'\left(\f{\ve_1}t\left(\f{E(t)}{E(0)}\right)^{\f{q+1}2}\right)R'(t)\\
& \le-\f\ep2E(t)^{\f{q+1}2}G'\left(\f{\ve_1}t\left(\f{E(t)}{E(0)}\right)^{\f{q+1}2}\right)+C\,t\,G'\left(\f{\ve_1}t\left(\f{E(t)}{E(0)}\right)^{\f{q+1}2}\right)G^{-1}\left(\f{\ga\,\mu(t)}{t\,\xi(t)}\right).
\end{align*}
Here we apply \eqref{x6.7} with
\[
A=G'\left(\f{\ve_1}t\left(\f{E(t)}{E(0)}\right)^{\f{q+1}2}\right),\quad B=G^{-1}\left(\f{\ga\,\mu(t)}{t\,\xi(t)}\right)
\]
to the above inequality to further bound
\begin{align*}
R_1'(t) & \le-\f\ep2E(t)^{\f{q+1}2}G'\left(\f{\ve_1}t\left(\f{E(t)}{E(0)}\right)^{\f{q+1}2}\right)+C\,t\,G^*\left(G'\left(\f{\ve_1}t\left(\f{E(t)}{E(0)}\right)^{\f{q+1}2}\right)\right)+C\f{\ga\,\mu(t)}{\xi(t)}\\
& \le-\f\ep2E(t)^{\f{q+1}2}G'\left(\f{\ve_1}t\left(\f{E(t)}{E(0)}\right)^{\f{q+1}2}\right)+C\ve_1\left(\f{E(t)}{E(0)}\right)^{\f{q+1}2}G'\left(\f{\ve_1}t\left(\f{E(t)}{E(0)}\right)^{\f{q+1}2}\right)+C\f{\ga\,\mu(t)}{\xi(t)}.
\end{align*}
Multiplying both sides of the above estimate by $\xi(t)$, we use \eqref{x3.28} and the facts that
\[
\f{\ve_1}t\left(\f{E(t)}{E(0)}\right)^{\f{q+1}2}<f(0)
\]
to derive
\begin{align}
\xi(t)R_1'(t) & \le-\f\ep2\xi(t)E(t)^{\f{q+1}2}G'\left(\f{\ve_1}t\left(\f{E(t)}{E(0)}\right)^{\f{q+1}2}\right)\nonumber\\
& \quad\,+C\xi(t)\ve_1\left(\f{E(t)}{E(0)}\right)^{\f{q+1}2}G'\left(\f{\ve_1}t\left(\f{E(t)}{E(0)}\right)^{\f{q+1}2}\right)+C\ga\,\mu(t)\nonumber\\
& \le-\ve_2\xi(t)\left(\f{E(t)}{E(0)}\right)^{\f{q+1}2}G'\left(\f{\ve_1}t\left(\f{E(t)}{E(0)}\right)^{\f{q+1}2}\right)-C E'(t),\label{6.9}
\end{align}
where
\[
\ve_2:=\f\ep2E(0)^{\f{q+1}2}-C\ve_1
\]
and we can choose $\ve_1>0$ sufficiently small such that $\ve_2>0$.

By further putting $R_2(t):=\xi(t)R_1(t)+C E(t)$ as $S_1(t)$ in Case 1, again we have $R_2(t)\sim E(t)$. Then we employ $\xi'(t)\le0$ and \eqref{6.9} to estimate $R_2'(t)$ as
\[
R_2'(t)=\xi'(t)R_1(t)+\xi(t)R_1'(t)+C E'(t)\le-\ve_2\xi(t)\left(\f{E(t)}{E(0)}\right)^{\f{q+1}2}G'\left(\f{\ve_1}t\left(\f{E(t)}{E(0)}\right)^{\f{q+1}2}\right)\le0.
\]
Then obviously $R_2'(t)\le0$ and integrating the above inequality over some interval $(t_0,t)\subset\BR_+$ yields
\[
\ve_2\int_{t_0}^t\xi(s)\left(\f{E(s)}{E(0)}\right)^{\f{q+1}2}G'\left(\f{\ve_1}s\left(\f{E(s)}{E(0)}\right)^{\f{q+1}2}\right)\rd s\le-\int_{t_0}^t R_2'(s)\,\rd s\le R_2(t_0).
\]
By the monotonicity of $G'(t),G''(t)$ and $E(t)$, the function
\[
t\longmapsto\left(\f{E(t)}{E(0)}\right)^{\f{q+1}2}G'\left(\f{\ve_1}t\left(\f{E(t)}{E(0)}\right)^{\f{q+1}2}\right)
\]
turns out to be nonincreasing. Therefore, we arrive at
\[
\ve_2\left(\f{E(t)}{E(0)}\right)^{\f{q+1}2}G'\left(\f{\ve_1}t\left(\f{E(t)}{E(0)}\right)^{\f{q+1}2}\right)\int_{t_0}^t\,\xi(s)\,\rd s\le R_2(t_0).
\]
Dividing both sides of the above estimate by $t$ and recalling the definition of $G_2(t)$ in Theorem \ref{the3.2}, we conclude
\[
\ve_2G_2\left(\f1t\left(\f{E(t)}{E(0)}\right)^{\f{q+1}2}\right)\int_{t_0}^t\,\xi(s)\,\rd s\le\f{R_2(t_0)}t.
\]
Consequently, owing to the fact that $G_2$ is strictly increasing in $\ov{\BR_+}$, we obtain the desired estimate \eqref{rr3.44} by putting $k_2:=R_2(t_0)/\ve_2$. The proof of Theorem \ref{the3.2} is completed.

\section*{Acknowledgements}

The first author is supported by China Scholarship Council. The second author is supported by JSPS KAKENHI Grant Numbers JP22K13954 and JP23KK0049.


\begin{thebibliography}{00}

\bibitem{A3}
R.A. Adams and J.F. Fournier, Sobolev Spaces (2nd Ed.), Academic Press, New York, 2003.

\bibitem{5}
V.I. Arnold, Mathematical Methods of Classical Mechanics (2nd Ed.), Springer, New York, 1989.

\bibitem{C2}
M.M. Cavalcanti, V.N.D. Cavalcanti, I. Lasiecka and F.A. Nascimento, Intrinsic decay rate estimates for the wave equation with competing viscoelastic and frictional dissipative effects, {\it Discrete Contin. Dyn. Syst. Ser. B}, {\bf19}(7), 2014, 1987--2012.

\bibitem{C1}
M.M. Cavalcanti and H.P. Oquendo, Frictional versus viscoelastic damping in a semilinear wave equation, {\it SIAM J. Control Optim.}, {\bf42}(4), 2003, 1310--1324.

\bibitem{cui}
J. Cui and S. Chai, Energy decay for a wave equation of variable coefficients with logarithmic nonlinearity source term, {\it Appl. Anal.}, {\bf102}(6), 2023, 1696--1710.

\bibitem{D1}
H. Di, Y. Shang and Z. Song, Initial boundary value problem for a class of strongly damped semilinear wave equations with logarithmic nonlinearity, {\it Nonlinear Anal. Real World Appl.}, {\bf51}, 2020, 102968 (22pp).

\bibitem{D3}
F.R. Dias Silva, F.A.F. Nascimento and J.H. Rodrigues, General decay rates for the wave equation with mixed-type damping mechanisms on unbounded domain with finite measure, {\it Z. Angew. Math. Phys.}, {\bf66}, 2015, 3123--3145.

\bibitem{G1}
F. Gazzola and M. Squassina, Global solutions and finite time blow up for damped semilinear wave equations, {\it Ann. Inst. H. Poincar\'e C Anal. Non Lin\'eaire}, {\bf23}(2), 2006, 185--207.

\bibitem{H1}
T.G. Ha and S.-H. Park, Blow-up phenomena for a viscoelastic wave equation with strong damping and logarithmic nonlinearity, {\it Adv. Difference Equ.}, 2020, 235 (17pp).

\bibitem{H2}
J. Hao and F. Du, Decay and blow-up for a viscoelastic wave equation of variable coefficients with logarithmic nonlinearity, {\it J. Math. Anal. Appl.}, {\bf506}(1), 2022, 125608 (20pp).

\bibitem{22}
K.-P. Jin, L. Jin and T.-J. Xiao, Uniform stability of semilinear wave equations with arbitrary local memory effects versus frictional dampings, {\it J. Differential Equations}, {\bf266}(11), 2019, 7230--7263.

\bibitem{F1}
F. Li and Q. Gao, Blow-up of solution for a nonlinear Petrovsky type equation with memory, {\it Appl. Math. Comput.}, {\bf274}, 2016, 383--392.

\bibitem{L1}
W. Lian and R. Xu, Global well-posedness of nonlinear wave equation with weak and strong damping terms and logarithmic source term, {\it Adv. Nonlinear Anal.}, {\bf9}(1), 2020, 613--632.

\bibitem{23}
M. Liao, B. Guo and X. Zhu, Energy decay rates of solutions to a viscoelastic wave equation with variable exponents and weak damping, preprint, arXiv:2011.11185v1.

\bibitem{M1}
L. Ma and Z.B. Fang, Energy decay estimates and infinite blow-up phenomena for a strongly damped semilinear wave equation with logarithmic nonlinear source, {\it Math. Methods Appl. Sci.}, {\bf41}(7), 2018, 2639--2653.

\bibitem{11}
M.I. Mustafa, Optimal decay rates for the viscoelastic wave equation, {\it Math. Methods Appl. Sci.}, {\bf41}(1), 2018, 192--204.

\bibitem{M2}
M.I. Mustafa and G.A. Abusharkh, Plate equations with frictional and viscoelastic dampings, {\it Appl. Anal.}, {\bf96}(7), 2017, 1170--1187.

\bibitem{P2}
Q. Peng and  Z. Zhang, Stabilization and blow-up in an infinite memory wave equation with logarithmic nonlinearity and acoustic boundary conditions, {\it J. Syst. Sci. Complex.}, {\bf37}(4), 2024, 1368--1391.

\bibitem{W1}
G.F. Webb, Existence and asymptotic behavior for a strongly damped nonlinear wave equation, {\it Canad. J. Math.}, {\bf32}(3), 1980, 631--643.

\bibitem{Y1}
Y. Yang and Z.B. Fang, On a strongly damped semilinear wave equation with time-varying source and singular dissipation, {\it Adv. Nonlinear Anal.}, {\bf12}, 2023, 20220267 (23pp).

\end{thebibliography}
\end{document}